\documentclass[letterpaper,10pt,pdflatex,reqno]{amsart}
\usepackage{amsmath,times}
\usepackage{amsthm}
\usepackage{amssymb}
\usepackage{latexsym}
\usepackage{amscd}
\usepackage{bbm}
\usepackage[english]{babel}
\usepackage[latin1]{inputenc}
\usepackage{graphicx}
\usepackage{color}
\usepackage{verbatim}

\newcommand{\C}{\mathbb{C}}

\newcommand{\Q}{\mathbb{Q}}
\newcommand{\R}{\mathbb{R}}

\newcommand{\cE}{\mathcal{E}}

\newcommand{\cL}{\mathcal{L}}

\newcommand{\cX}{\mathcal{X}}

\DeclareMathOperator{\MA}{MA}
\DeclareMathOperator{\Ric}{Ric}
\DeclareMathOperator{\DF}{DF}
\DeclareMathOperator{\DING}{Ding}
\DeclareMathOperator{\Ding}{Ding}

\DeclareMathOperator{\Aut}{Aut}

\DeclareMathOperator{\PSH}{PSH}

\numberwithin{equation}{section}       

\newtheorem{prop} {Proposition} [section]
\newtheorem{thm}[prop] {Theorem} 

\newtheorem{lem}[prop] {Lemma}

\newtheorem{prop-def}[prop]{Proposition-Definition}

\newtheorem*{thmA}{Theorem A} 
 
\newtheorem*{thmB}{Theorem B} 
\newtheorem*{thmC}{Theorem C} 
\newtheorem*{thmD}{Theorem D} 
\newtheorem*{corA}{Corollary A}

\theoremstyle{remark}

\newtheorem{theorem}{Theorem}[section]

\theoremstyle{definition}
\newtheorem{definition}[theorem]{Definition}

\newtheorem{remark}[theorem]{Remark}
\newtheorem{conjecture}[theorem]{Conjecture}

\begin{document}

\title{Coupled K\"ahler-Einstein metrics}
\author{Jakob Hultgren and David Witt Nystr\"om}
\maketitle

\begin{abstract}
We propose new types of canonical metrics on K\"ahler manifolds, called coupled K\"ahler-Einstein metrics, generalizing K\"ahler-Einstein metrics. We prove existence and uniqueness results in the cases when the canonical bundle is ample and when the manifold is K\"ahler-Einstein Fano. In the Fano case we also prove that existence of coupled K\"ahler-Einstein metrics imply a certain algebraic stability condition, generalizing $K$-polystability.
\end{abstract}

\section{Introduction}


Given a compact K\"ahler manifold $X$ and K\"ahler class $\alpha$ on $X$, the problem of finding a canonical metric in $\alpha$ was proposed by Calabi \cite{Calabi1, Calabi2}. In particular he suggested looking for metrics of constant scalar curvature. In the special case when $\alpha = \pm c_1(X)$, where $c_1(X)$ is the first Chern class of $X$, this coincides with K\"ahler-Einstein metrics, defined by the fact that they are proportional to their Ricci tensor. 
It was recently showed by Chen, Donaldson and Sun \cite{CDS} (see also \cite{CSW}, \cite{DS} and \cite{Tian15}) that existence of K\"ahler-Einstein metrics is equivalent to a certain algebraic stability condition originally proposed by Tian \cite{Tian97}: $K$-stability. The corresponding conjecture for constant scalar curvature metrics, relating existence to a similar algebraic stability condition is still open. Some of the complexity of the problem is illustrated by the fact that even on manifolds admitting a K\"ahler-Einstein metric there might be K\"ahler classes not containing constant scalar curvature metrics (see \cite{Ross}). In this paper we propose new types of canonical metrics generalizing K\"ahler-Einstein metrics. 
\begin{definition}
We say that a $k$-tuple of K\"ahler metrics $(\omega_i)_{1\leq i\leq k}$ is coupled K\"ahler-Einstein (cKE) if
\begin{equation} \Ric \omega_1 = \ldots = \Ric \omega_k = \pm \sum \omega_i. \label{eq:CKE} \end{equation}
\end{definition}
Note that if $k=1$, then $(\omega_1)$ is coupled K\"ahler-Einstein if and only if $\omega_1$ is a K\"ahler-Einstein metric. Recall also that if $\theta$ is a $(1,1)$-current on $X$, then a K\"ahler metric $\omega$ is a $\theta$-twisted K\"ahler-Einstein metric if 
$$ \Ric \omega = \omega + \theta. $$
Consequently, $(\omega_i)$ is a coupled K\"ahler-Einstein if, for each $i$, $\omega_i$ is a $\sum_{j\not= i} \omega_j$-twisted K\"ahler-Einstein metric.

Equation \eqref{eq:CKE} implies the cohomological condition
$$ c_1(X) = \pm\sum [\omega_i]. $$ 
A $k$-tuple of K\"ahler classes $(\alpha_i)_{1\leq i\leq k}$ such that 
$$ \pm c_1(X) = \sum \alpha_i$$
will be called a \emph{decomposition of $\pm c_1(X)$}.

\begin{thmA}
Assume that $K_X$ is ample and let $(\alpha_i)$ be a decomposition of $-c_1(X)$. Then there exists a unique coupled K\"ahler-Einstein $k$-tuple $(\omega_i)$ such that $([\omega_i])=(\alpha_i)$.
\end{thmA}
Note that given a K\"ahler class $\alpha$ there is a natural family of decompostion of $\pm c_1(X)$, namely
$$ \{(t\alpha,\pm c_1(X) - t\alpha\} $$
where $t$ ranges over all positive real numbers such that $\pm c_1(X) - t\alpha>0$. Moreover, if $(t\alpha,\pm c_1(X)-\alpha)$ admits a coupled K\"ahler-Einstein metric $(\omega,\omega')$ then, putting $\omega_t=\omega/t$ gives a metric in $\alpha$. Assuming the corresponding equations are solvable, this opens up the possibility to associate a family $\{\omega_t\}\subset \alpha$ of canonical metrics to a K\"ahler class $\alpha$. Moreover, in contrast to constant scalar curvature metrics the metrics in $\{\omega_t\}$ satisfies a Ricci curvature bound:
\begin{equation} \pm\Ric \omega_t = \pm\Ric \omega =\omega + \omega' > t\omega_t. \label{eq:riccibound} \end{equation}
Note that if $X$ is K\"ahler-Einstein then the Calabi-Yau theorem gives a special metric in any class $\alpha$, namely the unique $\omega\in \alpha$ satisfying 
$$ \Ric \omega = \omega_{KE} $$
where $\omega_{KE}$ is the unique K\"ahler-Einstein metric on $X$. It is interesting to note that, formally, this arises as a limiting case of the family $\{\omega_t\}$. To see this, note that $\omega_t$ satisfies
$$ \Ric \omega_t = \Ric \omega = \Ric \omega' = \pm(t\omega_t+\omega'). $$
Letting $t$ tend to zero in this gives
$$ \Ric \omega_t = \Ric \omega' = \pm\omega'. $$

As a corollary of Theorem~A we get
\begin{corA}
Assume that $K_X$ is ample and let $\alpha$ be a K\"ahler class on $X$. Then $\alpha$ contains a canonical family of metrics $\{\omega_t, t\in(0,t_0)\}$ where $t_0=\sup\{t>0: -c_1(X) - t\alpha>0\}$. Moreover, these metrics satisfy
\begin{equation} \Ric \omega_t < -t\omega_t. \label{eq:ricciboundcor} \end{equation}
\end{corA}
\begin{proof}
By Theorem~A the decomposition $(t\alpha,-c_1(X) - t\alpha)$ admits a coupled K\"ahler-Einstein tuple $(\omega,\omega')$. Putting $\omega_t=\omega/t\in \alpha$ we get \eqref{eq:ricciboundcor} as in \eqref{eq:riccibound}.
\end{proof}

Finally, we point out that in \cite{Ross} Ross exhibit classes $\alpha$ on K\"ahler manifolds $X$ such that $K_X$ is ample and $\alpha$ does not admit a constant scalar curvature metrics. This means that by Corollary~A there are manifolds with K\"ahler classes that does not admit constant scalar curvature metrics but that does admit families of canonical metrics as above. 

Moving on to the Fano case we present the following uniqueness result which, together with Theorem~A, ensures any coupled K\"ahler-Einstein k-tuple with respect to a certain decomposition of $\pm c_1(X)$ is essentially unique. 
\begin{thmB}
Assume $X$ is Fano and $(\omega_i)$ and $(\omega_i')$ are two coupled K\"ahler-Einstein $k$-tuples such that $([\omega_i])=([\omega_i'])$. Then there is $F\in \Aut_0(X)$ such that $(F^*\omega_i)=(\omega_i')$. In particular, if $\Aut(X)$ is discrete then any coupled K\"ahler-Einstein $k$-tuple representing $(\alpha_i)$ is unique. 
\end{thmB}

As expected, the existence question in the Fano case is complicated. The next theorem gives existence for a large family of decompositions on any K\"ahler-Einstein Fano manifold. Before we state it we fix the following terminology: A decomposition $(\alpha_i)$ of $\pm c_1(X)$ is called parallel if all $\alpha_i$ are proportional to $\pm c_1(X)$. It is clear that if $\omega_{KE}$ is a K\"ahler-Einstein metric in $\pm c_1(X)$ then for any collection of positive numbers $\lambda_i$ such that  $\sum_1^k \lambda_i = 1$ we have that $(\lambda_i\omega_{KE})$ is coupled K\"ahler-Einstein. 
\begin{thmC}
Assume that $X$ is Fano and K\"ahler-Einstein and that the automorphism group is discrete. Then for any decomposition $(\alpha_i)$ of $\pm c_1(X)$ which is close enough to being parallel there exists a unique coupled K\"ahler-Einstein $k$-tuple $(\omega_i)$ such that $([\omega_i])=(\alpha_i)$.
\end{thmC}

For general Fano manifolds we will formulate an algebraic stability condition for decompositions of $c_1(X)$ that generalizes $K$-polystability and which we conjecture is equivalent to the existence of coupled K\"ahler-Einstein $k$-tuples. Similarily as for $K$-polystability this stability condition will be defined in terms of invariants of test configurations. We therefore recall the definition of a test configuration.
\begin{definition}
Let $L$ be an ample $\Q$-linebundle over a projective manifold $X$. A test configuration $(\cX,\cL,\pi,\rho)$ for $(X,L)$ is a scheme $\cX$, a flat surjective morphism $\pi:\cX\rightarrow \C$ and a relatively ample $\Q$-line bundle $\cL$ together with a $\C^*$-action $\rho$ on $\cL$, compatible with the standard $\C^*$-action on $\C$, such that $(\cX_1,\cL|_1)=(X,L)$. We will further assume that the total space $\cX$ is normal.
\end{definition}
Now, the Donaldson-Futaki invariant of a test configuration $(\cX,\cL)$ is defined in the following way. Let $X_0$ be the reduction of the central fiber of $\cX$, $N_k$ be the dimension of $H^0(X_0,k\cL|_{X_0})$ and $w_k$ be the weight of the $\C^*$-action on the complex line 
$$\det H^0(X_0,k\cL|_{X_0}).$$ 
After normalization, $w_k$ has an expansion in powers of $1/k$ and the Donaldson-Futaki invariant is minus the subleading coefficient of this expansion:
$$ \frac{w_k}{kN_k} = c_0 - \frac{1}{2k}\DF(\cX,\cL) + O\left(\frac{1}{k^2}\right). $$
A pair $(X,L)$ is said to be $K$-polystable if $\DF(\cX,\cL)\geq 0$ for all test configurations $(\cX,\cL)$ associated to $(X,L)$, with equality if and only if $(\cX,\cL)$ is isomorphic to a product test configuration.  

We will be interested in an alternative formula for the Donaldson-Futaki invariant due to Phong-Ross-Sturm \cite{PRS}. We will state it here in the special case when $L=-K_X$. If $(\cX,\cL)$ is a test configuration for $(X,-K_X)$, let $\eta$ be the line bundle over $\C$ given by
\begin{equation} \eta =  \frac{n}{(n+1)(-K_X)^n}\left(\langle\cL,\ldots,\cL\rangle-(n+1)\langle-K_{\cX/\C},\cL,\ldots,\cL\rangle\right) \label{eq:standardeta} \end{equation}
where $\langle\cdot,\ldots,\cdot\rangle$ denotes Deligne pairing. Then the Donaldson-Futaki invariant is the weight of the $C^*$-action on the central fiber of $\eta$,
$$ \DF(\cX,\cL)=w_0(\eta). $$
This is closely related to the formula for the Donaldson-Futaki invariant in terms of intersection numbers given in \cite{Odaka},\cite{Wang}. 

A decomposition of $c_1(X)$ will be called a $\Q$-decomposition if it is of the form $(c_1(L_i))$ for some $\Q$-line bundles $L_1,\ldots,L_k$ over $X$.
\begin{definition}\label{def:testconfiguration}
Let $(\alpha_i)$ be a $\Q$-decomposition of $c_1(X)$. Assume $(\cX,\cL)$ is a test configuration for $(X,-K_X)$ that satisfies the following
\begin{itemize}
    \item $\cL = \sum \cL_i$ for some $\Q$-linebundles $\cL_i, 1\leq i\leq k$, over $\cX$ such that, for each $i$, $(\cX,\cL_i)$ is a test configuration for $(X,L_i)$ where $L_i$ is a $\Q$-line bundle over $X$. 
    \item $c_1(L_i)=\alpha_i$ for all $i$.
\end{itemize}
We then say that the data $(\cX,(\cL_i))$ is a test configuration for $(\alpha_i)$.
\end{definition}

Given a test configuration $(\cX,(\cL_i))$ for a $\Q$-decomposition of $c_1(X)$ let $\eta$ be the $\Q$-line bundle over $\C$ given by
\begin{eqnarray}\label{eq:eta}
\eta = -\frac{1}{(n+1)}\sum_i \frac{\langle \cL_i,\ldots,\cL_i\rangle}{(\cL_i|_{x_1})^n} +\frac{\langle \sum_i \cL_i+K_{\cX/\C}, \sum_i \cL_i,\ldots,\sum_i \cL_i \rangle}{(-K_X)^n}
\end{eqnarray}

\begin{definition}\label{def:Kstability}
Let $(\cX,(\cL_i))$ be a test configuration for a $\Q$-decomposition of $c_1(X)$. We define the Donaldson-Futaki invariant $\DF(\cX,(\cL_i))$ of $(\cX,(\cL_i))$ as the weight of the $\C^*$-action on the central fiber of $\eta$ (defined in \eqref{eq:eta}). 

Moreover, we say that a $\Q$-decomposition of $c_1(X)$ is \emph{$K$-polystable} if for all associated test configurations $(\cX,(\cL_i))$
$$\DF(\cX,(\cL_i))\geq 0,$$ 
with equality if and only if $\cX$ is isomorphic to $X\times \C$. 
\end{definition}
Note that if $k=1$, then $L_1=-K_X$ and \eqref{eq:eta} reduces to
$$ -\frac{1}{(n+1)(-K_X)^n}\langle \cL_1,\ldots,\cL_1\rangle + \frac{\langle \cL_1+K_{\cX/\C},\cL_1,\ldots, \cL_1\rangle}{(-K_X)^n}. $$
By bilinearity this equals \eqref{eq:standardeta}, hence in this case the generalized Donaldson-Futaki invariant and $K$-polystability of Definition~\ref{def:Kstability} reduces to the standard definitions.
With respect to this, we prove
\begin{thmD}
Assume $X$ is Fano. Then any decomposition of $c_1(X)$ that admits a coupled K\"ahler-Einstein $k$-tuple is $K$-polystable.
\end{thmD}
Moreover, we make the following
\begin{conjecture}
Assume $X$ is Fano. Then any decomposition of $c_1(X)$ that is $K$-polystable admits a coupled K\"ahler-Einstein $k$-tuple. 
\end{conjecture}

Recall that any K\"ahler class $\alpha$ on a Fano manifold comes with a family of decompositions $(t\alpha,c_1(X)-t\alpha)$ on $c_1(X)$. A subset of these will be $K$-polystable.  Conjecturally, these admit coupled K\"ahler-Einstein tuples that, as in Corolllary~A, determines a canonical family of metrics $\{\omega_t\}\subset \alpha$. Moreover, by \eqref{eq:riccibound} these would satisfy the lower Ricci curvature bounds
$$ \Ric \omega_t \geq t\omega_t. $$

\subsection{Outline}
Theorem~A is proved using a variational principle where coupled K\"ahler-Einstein metrics arise as the minimizers of a certain generalized Ding functional (see Section~\ref{sec:ding}). This is an adaptation of the variational approach to Monge-Amp\`ere equations developed by Berman, Boucksom, Guedj and Zeriahi \cite{BBGZ}. This variational principle is set up in Section~\ref{sec:variationalformulation} and the proof of Theorem~A is given in Section~\ref{sec:kxample}. To prove Theorem~B we combine the variational principle set up in Section~\ref{sec:variationalformulation} together with a convexity theorem by Berndtsson \cite{Berndtsson} (see Section~\ref{sect:uniqueness}). In Section~\ref{sect:Kstability}, we follow the approach developed by Berman in \cite{Berman}, to prove Theorem~D. In the final section we use the continuity method (see \cite{Aubin,Yau77,Yau78}) to prove Theorem~C.

\subsection{Acknowledgments}
We thank Robert Berman, Bo Berndtsson, Tam\'as Darvas, Mattias Jonsson, Valentino Tosatti and Vamsi Pritham Pingali for discussions relating to this work. The first named author was supported by the European Research Council and the second named author was supported by the Swedish Research Council.   

\section{Variational Formulation}
\label{sec:variationalformulation}
In this section we use the framework in \cite{BBGZ} to give a variational formulation for coupled K\"ahler-Einstein $k$-tuples. 

\subsection{Equivalent System of Monge-Amp\`ere Equations}

In this paper $c_1(X)$ will always be assumed to have a sign, and we will let $\lambda$ denote that sign, i.e. $X$ is either Fano and $\lambda:=1$, or else $K_X$ is ample in which case we set $\lambda:=-1$. 

Let $(\alpha_i)$ be a decomposition of $\pm c_1(X)$.  Fix one K\"ahler metric $\theta_i$ in each class $\alpha_i$. By Yau's theorem we can find a K\"ahler metric $\omega_0$ such that $\omega_0^n$ is a probability measure and 
$$ \Ric \omega_0 =\lambda \sum_i \theta_i. $$
For each $i$, let $\PSH(\theta_i)$ be the space of $\theta_i$-plurisubharmonic functions on $X$, in other words functions $\phi$ which are upper semicontinuous and locally integrable and which satisfies 
$$ dd^c \phi + \theta_i \geq 0. $$
We will consider these spaces equipped with the $L^1$-topology.

The Monge-Amp\`ere measure of a smooth $\theta_i$-plurisubharmonic function $\phi$ is given by
$$ \MA_{\theta_i}(\phi) := (dd^c\phi + \theta_i)^n. $$
An application of the maximum principle gives 
\begin{lem}
\label{lemma:MAeq}
A $k$-tuple of K\"ahler metrics $(dd^c\phi_i+\theta_i)$ is coupled K\"ahler-Einstein if and only if 
\begin{equation} \label{eq:MA}
\frac{1}{|\alpha_1|}\MA_{\theta_1}(\phi_1) = \ldots = \frac{1}{|\alpha_k|}\MA_{\theta_k}(\phi_k) = \frac{e^{-\lambda\sum \phi_i}\omega_0^n}{\int_X e^{-\lambda\sum \phi_i}\omega_0^n}
\end{equation}
where $|\alpha_i|:=\int_X \alpha^n$. 
\end{lem}

\subsection{Monge-Amp\`ere Energy}\label{sect:energy}
The \emph{Monge-Amp\`ere energy} of a smooth $\theta_i$-plurisub- harmonic functions $\phi$ is 
$$ E_{\theta_i}(\phi) = \frac{1}{n+1}\sum_{0\leq j\leq n} \int\phi (dd^c\phi+\theta_i)^j \wedge \theta_i^{n-j}. $$
It is clear that if $C\in \R$ then $E_{\theta_i}(\phi+C)=E_{\theta_i}(\phi)+C|\alpha_i|$. If $\phi$ is strictly $\theta_i$-plurisubharmonic and $v$ is a twice differentiable function on $X$, then $\phi+tv$ is also $\theta_i$-plurisubharmonic for small $t$ and one can verify that 
\begin{equation} 
\label{eq:Ederivative}
\left.\frac{d}{dt}\right|_{t=0} E_{\theta_i}(\phi+tv) = \int_Xv\MA(\phi) 
\end{equation}
Differentiating again shows that $E_{\theta_i}(\phi+tv)$ is concave in $t$ and integrating by parts gives the following:
\begin{lem}
\label{lemma:Eaffine}
If $\phi$ and $v$ are smooth functions on $X$ such that $\phi$ is strictly $\theta_i$-plurisub-harmonic and 
$$\left.\frac{d^2}{dt^2}\right|_{t=0}E_{\theta_i}(\phi+tv)=0$$
then $v$ is constant.
\end{lem}

The Monge-Amp\`ere energy has an extension to $\PSH(\theta_i)$ 
given by
$$ E_{\theta_i}(\phi) = \inf \{E_{\theta_i}(\psi): \psi\in \PSH(\theta_i) \textnormal{ smooth} \}. $$
We still have $E_{\theta_i}(\phi+C)=E_{\theta_i}(\phi)+C|\alpha_i|$ and By Proposition~2.1 in \cite{BBGZ} concavity is preserved. 

The \emph{finite energy class} $\cE^1(\theta_i)$ is the subset of $\PSH(\theta_i)$ where $E_{\theta_i}$ is finite. By concavity of $E_{\theta_i}$ $\cE^1(\theta_i)$ is convex. Moreover, $E_{\theta_i}$ is an exhaustion function of $\cE(\theta_i)$ in the sense that  
\begin{lem}[Lemma 2.6 in \cite{BBGZ}]
\label{lemma:compactness}
Let $\theta$ be a K\"ahler metric on a compact complex manifold $X$ and $C\in \R$. Then 
$$\{\phi\in \PSH(\theta):  \sup \phi \leq 0, E_{\theta}(\phi)\geq -C \}$$
is compact (in the $L^1$-topology).
\end{lem}


\subsection{Ding Functional}
\label{sec:ding}
\begin{definition}
Let $\theta_i$ and $\omega_0$ be as above. We define the associated Ding functional on $\prod_i \cE(\theta_i)$ by
\begin{equation}\label{eq:Ding}
D(\phi_1,\ldots,\phi_k) = -\sum_i \frac{1}{|\alpha_i|} E_{\theta_i}(\phi_i) -\lambda\log\int e^{-\lambda\sum_i \phi_i}\omega_0^n
\end{equation}
\end{definition}

We will prove the following statement.

\begin{thm}
\label{thm:solutionminimizer}
Assume $(\phi_i)\in \prod_i\cE^1(\theta_i)$ is a minimizer of the associated Ding functional. Then each $dd^c\phi_i+\theta_i$ is smooth and K\"ahler and the $k$-tuple $(dd^c\phi_i + \theta_i)$ is cKE.   
\end{thm}

Let $\phi\in \PSH(\theta_i)$ and $u$ be an upper semicontinuous function. The  $\theta_i$-plurisubharmonic envelope of $\phi+u$ is the function
$$ P(\phi+u)(x) = \sup \{\psi(x): \psi\in \PSH(\theta_i), \psi \leq \phi+u \}. $$
It is clear that $P(\phi+u)\leq \phi+u$ and $P(\phi+u)=\phi+u$ if $\phi+u\in \PSH(\theta_i)$. Moreover, $P(\phi+u)\in \cE^1(\theta_i)$ if $\phi\in \cE^1(\theta_i)$ (e.g. see \cite[Section 4.1]{BBGZ}). 

By the pioneering work of Bedford-Taylor \cite{BT} (in the local setting) and then Guedj-Zeriahi and others \cite{GZ,BEGZ} (for compact manifolds) it is possible to extend the definition of $\MA_{\theta_i}(\phi)$ to $\PSH(\theta_i)$. This gives a concept of weak solutions to~\eqref{eq:MA}. Moreover, one of the key point in the variational approach to Monge-Amp\`ere equations is given by the following characterization of this weak Monge-Amp\`ere operator.
\begin{thm}[\cite{BB},\cite{BBGZ}]
\label{thm:energydifferentiablity}
Let $\phi\in \cE^1(\theta_i)$ and $v$ be a continuous function on $X$. Then $E_{\theta_i}\circ P(\phi+tv)$ is differentiable at $t=0$ and
$$ \left. \frac{d}{dt}\right|_{t=0} E_{\theta_i}\circ P(\phi+tv) = \int_X v \MA(\phi). $$
\end{thm}
For $\phi$ bounded this was proved by Berman and Boucksom in \cite{BB}. For $\phi\in \cE^1(\theta_i)$ it was established in in \cite{BBGZ}.

Moreover, we have
\begin{lem}
\label{lemma:volumeterm}
Let $(\phi_i)\in \prod_i\cE^1(\theta_i)$ and let $(v_i)$ be a $k$-tuple of continuous functions on $X$. Then 
$$ -\lambda\log\int_X e^{-\lambda \sum \phi_i+tv_i}\omega_0^n$$ 
is differentiable at $t=0$ and
$$ \left. \frac{d}{dt}\right|_{t=0} -\lambda\log\int e^{-\lambda \sum \phi_i+tv_i}\omega_0^n = \frac{\int_X \left(\sum v_i\right) e^{-\lambda\sum \phi_i}\omega_0^n}{\int_X e^{-\lambda\sum \phi_i}\omega_0^n}. $$
\end{lem}
\begin{proof}
See e.g. \cite[Lemma 6.1]{BBGZ}. 
\end{proof}
Together with the previous lemma this gives
\begin{prop}
\label{prop:weaksolutionminimizer}
If $(\phi_i)\in \prod_i\cE^1(\theta_i)$ is a minimizer of the associated Ding functional then $(\phi_i)$ is a weak solution of \eqref{eq:MA}.
\end{prop}

\begin{proof}
Let $v_1,\ldots,v_k$ be continuous functions on $X$ and consider the function on $\R$ defined by
$$ G(t) = -\sum_i \frac{1}{|\alpha_i|}E_{\theta_i}\circ P(\phi_i+tv_i) - \lambda\log \int e^{-\lambda\sum \phi_i + tv_i}\omega_0^n. $$
Since $P(\phi_i + tv_i) \leq \phi_i + tv_i$ for all $i$ we get 
\begin{eqnarray}
G(t) & \geq & -\sum_i \frac{1}{|\alpha_i|} E_{\theta_i}\circ P(\phi_i+tv_i) - \lambda\log \int e^{-\lambda\sum P(\phi_i+tv_i)}\omega_0^n \nonumber \\
& = & D(P(\phi_i+tv_i)) \geq D(\phi_i) = G(0). \nonumber 
\end{eqnarray}
Thus $t=0$ is a minimizer for $G$. 
By Theorem~\ref{thm:energydifferentiablity} and Lemma~\ref{lemma:volumeterm} $G$ is differentiable at $0$ and 
\begin{equation}
\label{eq:dingderivative}
\left. \frac{d}{dt}\right|_{t=0} G(t) = -\sum_i \frac{1}{|\alpha_i|}\int_X v_i \MA(\phi_i) + \int_X \left(\sum v_i\right) e^{-\lambda\sum \phi_i}\omega_0^n / \int_X e^{-\lambda\sum \phi_i}\omega_0^n. 
\end{equation}
Since $0$ is a minimizer for $G$ this has to vanish. Fixing $1\leq j\leq k$ and putting $v_i=0$ for all $i\not=j$ gives
\begin{eqnarray}
0 & = & -\frac{1}{|\alpha_j|}\int_X v_j \MA(\phi_j) + \int_X v_j e^{-\lambda\sum \phi_i} \omega_0^n / \int_X e^{-\lambda\sum \phi_i}\omega_0^n \nonumber \\
& = & \int_X v_j\left(-\frac{1}{|\alpha_j|}\MA(\phi_j)+e^{-\lambda\sum \phi_i} \omega_0^n / \int_X e^{-\lambda\sum \phi_i}\omega_0^n\right) \nonumber
\end{eqnarray}
and since this holds for all continuous functions $v_j$ we get 
\begin{equation} 
\label{eq:MAj}
\frac{1}{|\alpha_j|}\MA(\phi_j) = e^{-\lambda\sum \phi_i} \omega_0^n / \int_X e^{-\lambda\sum \phi_i}\omega_0^n.
\end{equation}
By repeating the argument we see that \eqref{eq:MAj} holds for all $j$.
\end{proof}

\subsection{Regularity of Weak Solutions}
Here we will prove the following
\begin{thm}
\label{thm:regularity}
Assume $(\phi_i)\in \prod_i \cE(\theta_i)$ is a weak solution of \eqref{eq:MA}. Then $\phi_i$ is smooth for each $i$ and $(\phi_i)$ satisfies \eqref{eq:MA} in the classical sense, hence by Lemma \ref{lemma:MAeq} $(dd^c\phi_i+\theta_i)$ is cKE. 
\end{thm}
The proof will follow an argument in \cite{BBEGZ}. The crucial point is to obtain a priori estimates on $dd^c\phi_i+\theta_i$. Then the theorem will follow from theory for elliptic partial differential equations. Most of the literature we will quote below treat more general cases than what is needed here.

First of all there is an alternative characterization of finite energy metrics. Let $\omega$ be a K\"ahler metric. We say that $\phi\in \PSH(\omega)$ has full Monge-Amp\`ere mass if 
$$ \int_X \MA_{\omega}(\phi)=\int_X \omega^n.$$
\begin{prop}[Proposition 2.8 in \cite{BBGZ}]
\label{prop:fullmass}
Let $\phi\in \PSH(\omega)$. 
Then $\phi\in \cE^1(\omega)$ if and only if $\phi$ has full Monge-Amp\`ere mass and 
$$ \int |\phi| \MA(\phi)<\infty. $$
\end{prop}

We then have
\begin{thm}[Theorem 1.1 in \cite{BBEGZ}]
\label{thm:lelongnumbers}
Assume $\phi\in \PSH(\omega)$ has full Monge-Amp\`ere mass. Then $\phi$ has 
zero Lelong numbers.
\end{thm}
Note that, since any $\omega$-plurisubharmonic functions is bounded from above we get that $e^{\phi}\in L^p$ for all $p$. By a classical result of Skoda \cite{Skoda}, $\phi\in \PSH(\omega)$ has zero Lelong numbers if and only if $e^{-\phi}\in L^p$ for all $p<\infty$. This allows us to apply
\begin{thm}[Theorem C in \cite{EGZ}]
\label{thm:continuity}
Let $\omega$ be a K\"ahler metric and $\mu$ be a probability measure with $L^p$-density for $p>1$. Then the unique locally bounded $\omega$-plurisubharmonic function $\phi$ such that $$\frac{\MA_{\omega}(\phi)}{\int_X \omega^n}=\mu$$ and $\int \phi \omega^n=0$ is continuous. 
\end{thm}
The final ingredient in the proof will be
\begin{thm}[Theorem 10.1 in \cite{BBEGZ}]
\label{thm:laplaceestimate}
Let $\omega$ be a K\"ahler metric and $\mu$ be a probability measure on $X$ of the form $e^{\phi^+-\phi^-}dV$ with 
$dV$ a volume form, $\phi^\pm\in \PSH(\omega)$ and $e^{-\phi^-}\in L^p$ for some $p>1$. Assume $\phi\in \PSH(\omega)$ is bounded and solves $$\frac{\MA_{\omega}(\phi)}{\int_X \omega^n}=\mu,$$ then $\Delta \phi = O(e^{-\psi^-}).$
In particular, there is a constant $A$, depending only on $\sup \psi^+$, $\|e^{-\psi^-}\|$, $p$ and $\omega$ such that
$$ 0 \leq \omega + dd^c\phi \leq A e^{-\psi^-}\omega. $$
\end{thm}

\begin{proof}[Proof of Theorem~\ref{thm:regularity}]
By Proposition~\ref{prop:fullmass} and Theorem~\ref{thm:lelongnumbers} $\phi_i$ has zero Lelong numbers for all $i$. Equivalently $e^{\lambda\phi_i}\in L^p$ for all $p<\infty$ and $i$. We conclude that $$ e^{\lambda \sum \phi_i} \in L^p$$
for all $p<\infty$. By Theorem~\ref{thm:continuity} $\phi_i$ is continuous, and thus bounded, for all $i$. This means we may apply Theorem~\ref{thm:laplaceestimate} to get that there is a constant $A$ such that 
$$ 0 < \theta_i + dd^c \phi_i \leq A\theta_i. $$
The theorem then follows from general theory for elliptic partial differential equations. See the estimate in Theorem~1.2, Remark~1.4 in \cite{WangY} and the bootstrapping technique explained in Section~5.3 of \cite{Blocki} for details.
\end{proof}

\begin{proof}[Proof of Theorem~\ref{thm:solutionminimizer}]
By Proposition~\ref{prop:weaksolutionminimizer} any minimizer $(\phi_i)$ of $D$ is a weak solution of \eqref{eq:MA}. Theorem~\ref{thm:regularity} says that any weak solution of \eqref{eq:MA} is a classical solution which by Lemma~\ref{lemma:MAeq} means that $(dd^c\phi_i+\theta_i)$ is cKE. 
\end{proof}

\section{Existence and uniqueness when $K_X$ is ample}
\label{sec:kxample}
In this section we will prove Theorem A. By the variational principle set up in the previous section this reduces to proving that the Ding functional associated to any decomposition of $-c_1(X)$ admits a unique minimizer. We begin with

\begin{lem}
\label{lemma:dingbasic}
Assume $(\alpha_i)$ is a decomposition of $-c_1(X)$ and $(\theta_i)$ a $k$-tuple of K\"ahler metrics such that $([\theta_i])=(\alpha_i)$. Then 
the associated Ding functional is lower semicontinuous, convex and translation invariant. 
\end{lem}

\begin{proof}
As mentioned in Section~\ref{sect:energy} $E_{\theta_i}$ is concave and upper semicontinuous on $\PSH(\theta_i)$. Moreover, since any $\theta_i$-plurisubharmonic function is bounded from above
$$ \log\int_X e^{\sum\phi_i}\omega_0^n $$
is continuous by the Dominated Convergence Theorem. By H\"older's inequality it is also convex. This proves that $D$ is convex and lower semicontinuous. Translation invariance follows from $E_{\theta_i}(\phi+C)=E_{\theta_i}(\phi)+C|\alpha_i|$ and 
$$ \log \int_X e^{\sum \phi_i+C} \omega_0^n = \log \int_X e^{\sum \phi_i} \omega_0^n+C. $$
\end{proof}

The crucial point in the proof of Theorem A is the following coercivity estimate
\begin{lem}
\label{lemma:coercivity}
Let $(\alpha_i)$ be a decomposition of $-c_1(X)$ and $(\theta_i)$ be a $k$-tuple of K\"ahler metrics such that $([\theta_i])=(\alpha_i)$. Then there is a constant $C$ such that for all $(\phi_i)\in \prod_i \cE(\theta_i)$ satisfying
\begin{equation}
\label{eq:normalization}
\sup_X\phi_1 = \ldots = \sup_X \phi_k = 0.
\end{equation}
we have 
$$ D((\phi_i)) \geq -\frac{1}{|\alpha_j|}E_j(\phi_j) - C $$
for all $j$.
\end{lem}

\begin{proof}
As explained in the proof of Lemma~\ref{lemma:dingbasic}
\begin{equation} 
\label{eq:volumeterm}
\log\int_X e^{\sum \phi_i}\omega_0^n 
\end{equation}
is continuous in $(\phi_i)$. Moreover, by standard properties of plurisubharmonic functions
\begin{equation}
\label{eq:normalizedpsh}
\prod_i \PSH(\theta_i)\cap \{\sup \phi_1 = \ldots = \sup \phi_k = 0\}
\end{equation}
is compact. This means \eqref{eq:volumeterm} can be bounded from below on \eqref{eq:normalizedpsh} by some constant $-C$. By \eqref{eq:normalization}, $E_{\theta_i}(\phi_i)\leq 0$ for all $i$. We conclude that  
$$ D((\phi_i)) \geq -\frac{1}{|\alpha_j|}E_j(\phi_j) - C  $$
for all $j$.
\end{proof}

We are now ready for
\begin{proof}[Proof of Theorem A]
Let 
$$(\phi^j_i)=(\phi_i^j)_{1\leq i\leq k}^{1\leq j < \infty} $$ 
be a sequence in $\prod \cE^1(\theta_i)$ such that $ D((\phi_i^j)) $ is decreasing in $j$ and tends to $\inf_{\prod \cE^1(\theta_i)} D$ as $j\rightarrow \infty$. Any function in $\PSH(\theta_i)$ is bounded from above. By translation invariance we may assume that \eqref{eq:normalization} holds for all $i$ and $j$.  Since $D((\phi_i^j))$ is decreasing it is bounded from above by some constant $C$. By Lemma~\ref{lemma:coercivity} there is another constant $C'$ such that 
$$ D((\phi_i^j)) \geq -\frac{1}{|\alpha_l|} E_l(\phi_i^j) - C' $$
for all $l$ and $j$. This means
$$  \frac{1}{|\alpha_l|}E_l(\phi_i^j) \geq -C-C' $$
for all $l$ and $j$ and consequently, for all $j$, $(\phi_i^j)$ is in
$$ \prod_i\left\{ \phi\in \cE^1(\theta_i): \frac{1}{|\alpha_i|}E_{\theta_i}(\phi) > - C-C', \sup \phi = 0\right\}$$
which by Lemma~\ref{lemma:compactness} is compact. We may thus extract a subsequence of $(\phi_i^j)$ that converges to some $(\phi_i)\in \prod_i \cE^1(\theta_i)$. By lower semicontinuity of $D$, $(\phi_i)$ is a minimizer of $D$. By Theorem~\ref{thm:solutionminimizer} this means $(dd^c\phi_i+\theta_i)$ is cKE. 

For uniqueness, assume $(\omega_i)$ and $(\omega_i')$ are cKE and $([\omega_i])=([\omega_i'])=(\alpha_i)$. Then $(\omega_i) = (dd^c\phi_i+\theta_i)$ and $(\omega_i')=(dd^c\phi_i'+\theta_i)$ for some $(\phi_i),(\phi_i')\in \prod_i\PSH(\theta_i)$ which are solutions to \eqref{eq:MA} by Lemma~\ref{lemma:MAeq} and smooth by Theorem~\ref{thm:regularity}. We may choose $(\phi_i)$ and $(\phi_i')$ so that
$$ \int e^{\sum \phi_i}\omega_0^n = \int e^{\sum \phi_i'}\omega_0^n = 1. $$
Let $\phi_i(t)=t\phi_i+(1-t)\phi_i'$. Then $D((\phi_i(t)))$ is convex in $t$ by Lemma~\ref{lemma:dingbasic}. Moreover, by \eqref{eq:Ederivative} we have, since $(\phi_i')$ is a solution of \eqref{eq:MA}, that 
\begin{eqnarray} 
\left.\frac{d}{dt} \right|_{t=0} D((\phi_i(t))) & = & -\sum_i \frac{1}{|\alpha_i|}\int_X (\phi_i-\phi_i') \MA(\phi_i') + \int_X \sum_i(\phi_i-\phi_i')e^{\sum \phi_i'} \omega_0^n \nonumber \\
& = & \sum_i \int_X (\phi_i-\phi_i') \left(-\frac{1}{|\alpha_i|}\MA(\phi_i')+e^{\sum \phi_i'} \omega_0^n\right) = 0. \nonumber
\end{eqnarray}
Similarly, since $(\phi_i)$ is also a solution of \eqref{eq:MA},  $\left.\frac{d}{dt}\right|_{t=1} D((\phi_i(t))) = 0$. Together this implies that $D((\phi_i(t)))$ is constant for $t\in [0,1]$. Since $D$ is a sum of convex terms this means each term has to be affine. In particular $E_{\theta_i}(\phi_i(t))$ is affine in $t$ for each $i$. By Lemma~\ref{lemma:Eaffine} this means $\phi_i = \phi_i'+C$ and thus $(\omega_i)=(\omega_i')$.  
\end{proof}

\begin{remark}
After a preprint of this paper appeared on arXiv an alternative proof of Theorem A using the continuity method was provided by Pingali \cite{Pingali}. Similarly as for usual K\"ahler-Einstein metrics, he also reduces the existence in the Fano case to a $C^0$-estimate. 
\end{remark}




\section{Uniqueness up to automorphisms when $X$ is Fano}
\label{sect:uniqueness}
In this section we prove Theorem~B. The proof is based on a result by Berndtsson stating that a certain functional on $c_1(X)$ is essentially strictly convex along geodesics. We begin by recalling the definition of geodesic segments and geodesic rays in the space of K\"ahler metrics.  

\begin{definition}
Assume $\theta$ is a K\"ahler form on $X$. Let $\Delta^*$ be the punctured unit disc in $\C$ and $A\subset \Delta^*$ be the annulus of outer radius 1 and inner radius $e^{-1}$. Abusing notation we will identify $\theta$ with its lift to the product $X\times A$ (or $X\times \Delta^*$). Let $\phi$ be an upper semi-continuous, locally integrable, $S^1$-invariant function on $X\times A$ (or $X\times \Delta^*$). Let $\Omega$ denote the $(1,1)$-form on $X\times A$ (or $X\times \Delta^*$) given by $\theta+dd^c\phi$ and for each $\tau\in A$ (or $\tau\in \Delta^*$), let $\phi(\tau)$ denote the restriction of $\phi$ to $X\times\{\tau\}$ and $\omega(\tau)$ denote the $(1,1)$-form $dd^c\phi+\theta$ on $X\times\{\tau\}$. Then $\phi$ is a weak geodesic segment (or a weak geodesic ray) in $\PSH(\theta)$ if
\begin{equation} \label{eq:subgeod} \omega(\tau) \geq 0 \end{equation}
for all $\tau$ and
\begin{equation} \label{eq:geod} \Omega^{n+1}=0. \end{equation}
We say that $\phi$ is a subgeodesic if \eqref{eq:subgeod} but not necessarily \eqref{eq:geod} holds. 
\end{definition}
In either case, since $\phi$ is assumed to be $S^1$-invariant we will also use the real logarithmic notation $\phi^t=\phi(\tau)$ where $t=-\log|\tau|^2$.

Let $\psi^0$ and $\psi^1$ be two smooth elements in $\PSH(\theta)$. By an envelope construction there is a unique weak geodesic segment $\psi$ such that $\phi^0=\psi^0$ and $\phi^1=\psi^1$. Although $\phi$ may not be smooth it is guaranteed, by a theorem of Chen \cite{Chen}, with contributions from Blocki \cite{}, that all mixed complex second derivatives of $\phi$ are bounded. 

We have the following 
\begin{lem}[Proposition~6.2 in \cite{BBGZ}] 
\label{lemma:eaffine}
Let $\phi$ be a locally bounded weak geodesic segment or a weak geodesic ray. Then the function 
$$ t\mapsto E_\theta(\phi^t) $$
is affine.
\end{lem}

The main ingredient in the proof of Theorem~B will be the following convexity theorem by Berndtsson. 
\begin{thm}[Theorem~1.1, Theorem~1.2 and Lemma 4.3 in \cite{Berndtsson}]
\label{thm:fanoconvex}
Assume $X$ is Fano, $\theta$ is a K\"ahler form in $c_1(X)$ and $\phi$ is a locally bounded subgeodesic in $\PSH(\theta)$. Let $\omega_\theta$ be a K\"ahler form on $X$ such that $\Ric \omega_\theta = \theta$. Then the function $$\mathcal{F}(t):=-\log\int_X e^{-\phi^t}\omega_\theta^n$$ 
is convex in $t$. If $\mathcal{F}(t)$ is affine then there exists a holomorphic vector field $V$ on $X$ such that 
$$ \left.\left(V-\frac{\partial}{\partial \tau}\right)\right\rfloor \Omega = 0. $$
\end{thm}

Assume $X$ is Fano and $\theta_1,\ldots, \theta_k$ are K\"ahler forms on $X$. Let $(\psi^0_i)$ and $(\psi^1_i)$ be two $k$-tuples in $\prod \PSH(\theta_i)$. Then there is a $k$-tuple of weak geodesic segments $(\phi_i)$ such that $\phi_i^0=\psi^0$ and $\phi_i^1=\psi_i^1$ for all $i$. We get $k$-tuples of $(1,1)$-forms $(\Omega_i)$ and $(\omega^t_i)$ on $X\times A$ and $X\times \{\tau\}$ respectively. 
\begin{lem}
\label{lemma:flowdec}
Assume $X$ is Fano and $(\phi_i)$ is a k-tuple of weak geodesic segments such that $\phi_i^0$ and $\phi_i^1$ are smooth for all $i$. 
Assume also that 
\begin{equation} \left.\left(V-\frac{\partial}{\partial \tau}\right)\right\rfloor \sum_i\Omega_i \label{eq:contraction} \end{equation}
for some holomorphic vector field $V$ on $X$. 
Then 
\begin{equation} F_t^*(\omega_i^t) = \omega_i^0 \label{eq:flowlemmai} \end{equation}
for all $i$ and $t$.
\end{lem}
\begin{proof}
By positivity, \eqref{eq:contraction} implies 
\begin{equation}  \left.\left(V-\frac{\partial}{\partial \tau}\right)\right\rfloor \Omega_i \label{eq:contractioni} \end{equation}
for all $i$.
We claim that this implies $\phi_i$ is smooth for all $i$. To see this note that \eqref{eq:contractioni} implies $\phi_i$ is harmonic along the leafs of the foliation induced by the vector field $V-\frac{\partial}{\partial \tau}$. This means $\phi_i$ is $\theta_i$-harmonic on these leaves. Moreover, each leaf is a complex annulus so the value of $\phi_i$ in any point can be recaptured, using the Green-Riesz representation formula, from the boundary data $\phi^0_i$ and $\phi^1_i$. Since this boundary data varies smoothly, the foliation is smooth and $\theta_i$ is smooth this implies $\phi_i$ is smooth. 

The derivative of the left hand side of \eqref{eq:flowlemmai} along $\frac{\partial}{\partial \tau}$ is the pullback under $F_\tau$ of the Lie derivative of $\Omega_i$ along $\frac{\partial}{\partial \tau}$, restricted to the fiber $X\times \{\tau\}$.
By Cartan's magic formula and \eqref{eq:contractioni} this vanishes since $\Omega_i$ is closed. This proves the lemma.
\end{proof}

\begin{thmB}
Assume $X$ is Fano and $(\omega_i)$ and $(\omega_i')$ are two coupled K\"ahler-Einstein $k$-tuples such that $([\omega_i])=([\omega_i'])$. Then there is $F\in \Aut_0(X)$ such that $(F^*\omega_i)=(\omega_i')$. In particular, if $\Aut(X)$ is discrete then any coupled K\"ahler-Einstein $k$-tuple representing $(\alpha_i)$ is unique. 
\end{thmB}
\begin{proof}
Let $(\psi_i)\in \Pi\cE(\omega_i)$ satisfy $\omega_i+dd^c\psi_i=\omega_i'$ and for each $i$, let $(\phi_i)$ be the $k$-tuple of geodesic segments such that $\phi_i^0=0$ and $\phi^1_i=\psi_i$. By Theorem~\ref{thm:fanoconvex} and Lemma~\ref{lemma:eaffine} $D(\phi_i^t)$ is convex in $t$. Since it is also stationary at $t=0$ and $t=1$ this implies it is affine. By Lemma~\ref{lemma:eaffine} this means
$$ -\log\int_X e^{-\sum\phi_i^t}\omega_0^n $$
is affine in $t$. Since $\sum \phi_i$ is a subgeodesic in $\PSH(\sum\omega_i)$ we may apply Theorem~\ref{thm:fanoconvex} and Lemma~\ref{lemma:flowdec} to conclude the theorem. 
%
\end{proof}

\section{K-polystability}\label{sect:Kstability}
In this section we prove Theorem~D. We begin by recalling the definition of $K$-polystability of a decomposition of $c_1(X)$ and introducing a second invariant which will play the same role as the Ding invariant in \cite{Berman}. 
%
\begin{definition}
Let $(\alpha_i)_{1\leq i\leq k}$ be a $\Q$-decomposition of $c_1(X)$. Assume $(\cX,\cL)$ is a test configuration for $(X,-K_X)$ that satisfies the following
\begin{itemize}
    \item $\cL = \sum \cL_i$ for some $\Q$-linebundles $\cL_i, 1\leq i\leq k$ over $\cX$ such that, for each $i$, $(\cX,\cL_i)$ is a test configuration for $(X,L_i)$ where $L_i$ is a $\Q$-linebundle over $X$. 
    \item $c_1(L_i)=\alpha_i$ for all $i$.
\end{itemize}
We then say that the data $(\cX,(\cL_i))$ is a test configuration for $(\alpha_i)$.
\end{definition}

Given a test configuration $(\cX,(\cL_i))$ for a $\Q$-decomposition of $c_1(X)$ we let $\eta$ be the $\Q$-line bundle over $\C$ given by
\begin{eqnarray}
\eta = -\frac{1}{(n+1)}\sum_i \frac{\langle \cL_i,\ldots,\cL_i\rangle}{(\cL_i|_{x_1})^n} +\frac{\langle \sum \cL_i+K_{\cX/\C}, \sum \cL_i,\ldots,\sum \cL_i \rangle}{(-K_X)^n}.
\end{eqnarray}
Moreover, we define the \emph{Ding line bundle} as the $\Q$-line bundle over $\C$ given by
\begin{eqnarray}\label{eq:delta}
\delta = -\frac{1}{(n+1)}\sum_i \frac{\langle \cL_i,\ldots,\cL_i\rangle}{(\cL_i|_{x_1})^n} + \pi_*\left(\sum_i \cL_i+K_{\cX/\C}\right).  
\end{eqnarray}

\begin{definition}
Let $(\cX,(\cL_i))$ be a test configuration for a $\Q$-decomposition of $c_1(X)$. We define the \emph{Donaldson-Futaki invariant} $\DF(\cX,(\cL_i))$ and the \emph{Ding invariant} $\DING(\cX,(\cL_i))$ of $(\cX,(\cL_i))$ as the weight of the $\C^*$-action on the central fibers of $\eta$ and $\delta$ respectively.  
Moreover, we say that a $\Q$-decomposition of $c_1(X)$ is \emph{$K$-polystable} if for all associated test configurations $(\cX,(\cL_i))$
$$\DF(\cX,(\cL_i))\geq 0,$$ 
with equality if and only if $\cX$ is isomorphic to $X\times \C$. 
\end{definition}

The first point in the proof of Theorem~D will be
\begin{lem}
\label{lemma:dfding}
We have $\DF(\cX,(\cL_i)) \geq \DING(\cX,(\cL_i))$. Moreover, if equality holds then $\cX$ is $\Q$-Gorenstein and $\sum \cL_i$ is isomorphic to $-K_{\cX/\C}$.  
\end{lem}
\begin{proof}
Note that 
$$ \eta = \delta + \frac{\langle \sum \cL_i+K_{\cX/\C},\sum \cL_i,\ldots,\sum \cL_i\rangle}{(-K_X)^n} - \pi_*\left(\sum_i \cL_i+K_{\cX/\C}\right). $$
Since $(\cX,\sum\cL_i)$ is a test configuration for $(X,-K_X)$ the first point of the lemma follows from Lemma~3.10 in \cite{Berman}. The second point follows from the proof of Theorem~3.11 in \cite{Berman}. 
\end{proof}

Let $L$ be an ample line bundle on $X$. If we pick a positive metric $h$ of $L$ with curvature $\theta$, then a function $\phi$ is $\theta$-plurisubharmonic if and only if $he^{-\phi/2}$ is a singular positive metric on $L$. In this section we will identify metrics on line bundles with quasi-plurisubharmonic functions in this manner. 

Let $(\cX,\cL)$ be a test configuration for a pair $(X,L)$ and $\phi^1$ a metric on $L$. Let $\Delta$ be the closed unit disc in $\C$ and $M=\pi^{-1}(\Delta)\subset \cX$. By identifying $\cL_1\rightarrow \cX_1$ with $L\rightarrow X$ and using the $\C^*$-action we may identify $\phi^1$ with the corresponding $S^1$-invariant metric on $\partial M=\pi^{-1}(S^1)$. Then a well known envelope construction defines a locally bounded metric $\phi$ on $\cL|_M$ such that its restriction to $M\setminus \pi^{-1}(0)$ defines a weak geodesic ray and, letting $\phi(\tau)$ for $\tau\in \Delta$ be the restriction of $\phi$ to $\cL|_\tau$, $\phi(1)=\phi^1$.  
(see Section~2.4 and Propoposition~2.7 in \cite{Berman}).
This means that a test configuration $(\cX,(\cL_i))$ for $(\alpha_i)$ and a $k$-tuple $(\phi_i^1)\in \Pi\cE^1(\theta_i)$ such that $\phi_i^1$ is locally bounded for each $i$ induces a $k$-tuple $(\phi_i)$, where each $\phi_i$ is a metric on $\cL_i|_M$ satisfying the properties above. Moreover, each $\phi_i$ define a family of metrics on $\cL_i|_1 = L_i$ by considering $\{\rho(\tau)^*\phi_i(\tau):\tau\in \Delta\}$. Since $\phi_i$ is $S^1$-invariant we will, as in Section~\ref{sect:uniqueness}, use real logarithmic notation and put 
$$ \phi_i^t := \rho(\tau)^*\phi_i(\tau) $$ 
where $t = -\log |\tau|^2$. We get a function on $[0,\infty)$
\begin{equation}
\label{eq:dingray}
t\mapsto D(\phi_i^t) 
\end{equation}
which by Theorem~\ref{thm:fanoconvex} and Lemma~\ref{lemma:eaffine} is convex. 

On the other hand, the $k$-tuple $(\phi_i)$ induces a metric on $\delta$ which we will call the \emph{Ding metric}. To define this metric, let $\Phi_i$ be the Deligne metric on $\langle \cL_i,\ldots,\cL_i \rangle$ induced by $\phi_i$. Let $\Phi_0$ be the $L^2$-type metric on $\pi_*(\sum \cL_i + K_{\cX/\C})$ induced by the metric $\sum \phi_i$ on $\sum \cL_i$ (see Section~3.1 in \cite{Berman}). We define the Ding metric on $\delta$ as
$$ \Phi = -\frac{1}{n+1}\sum \frac{\Phi_i}{L_i^n} + \Phi_0. $$ 
\begin{prop}
\label{prop:dingmetric}
Let $(\cX,(\cL_i))$ be a test configuration associated to a decomposition of $c_1(X)$. For each $i$, assume $\phi_i$ is a metric on $\cL_i|_M$ such that $(dd^c\phi_i)^{n+1} = 0$. Then the Ding metric associated to $(\phi_i)$ has positive curvature current. Moreover, if $\cX$ is $\Q$-Gorenstein, $\sum \cL_i = -K_{\cX/\C}$ and the curvature of the Ding metric vanishes near $0$, then $\cX$ is isomorphic to $X\times \C$.
\end{prop}

\begin{proof}
The proposition is proved in exactly the same way as Proposition~3.5 in \cite{Berman}. First of all, the curvature of the Deligne metrics $\Phi_i$ vanishes when $\phi_i^t$ are weak geodesic rays. Moreover, $(\cX,\sum \cL_i)$ is a test configuration for the pair $(X,-K_X)$ and $\sum \phi^t_i$ is a positive metric on $\sum \cL_i$. By Lemma~3.2 in \cite{Berman} the induced $L^2$-type metric on $\pi_*(\sum\cL_i+K_{\cX/\C})$ has positive curvature current. This proves the first part of the proposition. The second part follows from Proposition~3.3 in \cite{Berman}.  
\end{proof}

As in \cite{Berman} we may relate the Ding metric associated to $(\phi_i)$ to  \eqref{eq:dingray}.
\begin{lem}
\label{lemma:dingmetric}
For each $i$, let $\psi_i$ be the metrics on $L_i=\cL_i|_1$ with curvature $\theta_i$ and $\sigma_i$ be a trivializing section of $-\langle \cL_i,\ldots,\cL_i \rangle/L_i^n(n+1)$. Moreover, let $s$ be a trivializing section of $\pi_*(\sum \cL_i + K_{\cX/\C})$ and 
$$ S_1=\sigma_1\otimes\ldots\otimes \sigma_k \otimes s. $$
Then 
\begin{eqnarray} 
-\log\|\rho(\tau)S_1\|_\Phi^2 + \sum_i\log\|\sigma_i\|_{\psi_i}^2 & = &  -\sum_i\frac{1}{L_i^n}E_{\theta_i}(\phi_i^t) - \log \int_X e^{-\sum\phi_i^t} \omega_0^n \nonumber \\
& = & D(\phi_i^t). \nonumber
\end{eqnarray}
\end{lem}

\begin{lem}[Lemma~2.6 in \cite{Berman}]
\label{lemma:pshlimit}
Let $F$ be a line bundle over $\C$ equipped with a $\C^*$-action $\rho$ compatible with the standard one on $\C$. Fix an $S^1$-invariant metric $\Phi$ on $F|_\Delta$ with positive curvature current. Then the weight $w_0$ of the $\C^*$-action on the central fiber $F|_0$ is given by the following formula
$$ w_0 = - \lim_{t\rightarrow 0}\frac{d}{dt}\log \| \rho(\tau)S_1\|^2_\Phi + l_0(\Phi) $$
where $l_0(\Phi)$ is the Lelong number of the metric $\Phi$ at $0$, $t=-\log|\tau|^2$ and $S_1$ is a fixed non-zero point in $F|_1$.
\end{lem}


\begin{proof}[Proof of Theorem~D]
Assume $(\omega_i)$ is a cKE $k$-tuple such that $([\omega_i])=(\alpha_i)$. Let $(\phi_i^1) \in \prod \cE^1(\theta_i)$ satisfy $dd^c\phi_i^1+\theta_i=\omega_i$ and let $(\phi_i)$ be the induced $k$-tuple of metrics on $\cL_i$. By Lemma~\ref{lemma:dingmetric} and and Lemma~\ref{lemma:pshlimit} we have
$$ \Ding(\cX,(\cL_i)) = w_0(\delta) = \lim_{t\rightarrow \infty} D(\phi^t_i) + l_0(\Phi). $$ 
By Proposition~\ref{prop:dingmetric} $\Phi$ has positive curvature. This means $l_0(\Phi)\geq 0$. Moreover, since $(\omega_i)$ is cKE we have that $(\phi_i^1)$ minimizes $D$. This means, by convexity of $D$ along weak geodesics, that $D(\phi^t_i) \geq 0$ for $t>0$. We conclude that $\Ding(\cX,(\cL_i))$ is non-negative and hence, by Lemma~\ref{lemma:dfding}, $\DF(\cX,(\cL_i))$ is non-negative. 

Assume $\DF(\cX,(\cL_i))$ vanishes. Then $\Ding(\cX,(\cL_i))$ vanishes and by Lemma~\ref{lemma:dfding} $\cX$ is $\Q$-Gorentstein and $\sum \cL_i$ is isomorphic to $K_{\cX/\C}$. Moreover, $\Ding(\cX,(\cL_i))=0$ implies $\lim_{t\rightarrow \infty} D(\phi^t_i)=0$. Since $D$ is convex and $D(\phi^t_i) \geq 0$ for $t>0$ we conclude that $D(\phi^t_i)$ is affine in $t$. By the second point in Proposition~\ref{prop:dingmetric} this implies $\cX$ is isomorphic to $X\times\C$. 
\end{proof}

\section{Existence when $X$ is Fano}
In this section we prove Theorem~C. We will work with the non-normalized system of Monge-Amp\`ere equations
\begin{equation}
\frac{1}{|\alpha_1|}\MA_{\theta_1}(\phi_1) = \ldots = \frac{1}{|\alpha_k|}\MA_{\theta_k}(\phi_k) = e^{-\sum \phi_i}\omega_0^n
\label{eq:mann}
\end{equation}
This equation is equivalent to \eqref{eq:MA} in the sense that $(\phi_i)$ is a solution to \eqref{eq:MA} if and only if $(\phi_i+C_i)$ where 
$$ \sum_i C_i=\log\int_X e^{-\sum \phi_i}\omega_0^n $$
is a solution to \eqref{eq:mann}. 


\begin{definition}
Let $(\alpha_i)$ be a decomposition of $c_1(X)$ and $\theta_i\in \alpha_i$ for all $i$. 
Moreover, let $\Omega^{(1,1)}$ be the space of smooth $(1,1)$-forms on $X$ and 
$$ H_k = \left\{\eta=(\eta_i)\in \left(\Omega^{(1,1)}\right)^k: \eta_i \textnormal{ closed, } \sum_i [\eta_i] = 0 \right\}. $$
Let $U$ be the open subset of $H_k\times\prod\PSH(\theta_i)\cap C^{2,1/2}$ given by
$$ U = \left\{ (\eta,(\phi_i))\in H_k\times\prod_i\PSH(\theta_i)\cap C^{2,1/2}: \theta_i+\eta_i+\phi_i > 0 \textnormal{ for all i.} \right\}
$$
Given $\eta\in H_k$, let $\omega_\eta$ be the unique element in $c_1(X)$ such that 
$$ \Ric \omega_\eta = \sum_i \theta_i+\eta_i. $$
We define
$$ F:U \rightarrow \left(C^{0,1/2}\right)^k\times \R^{k-1} $$
as
$$ F(\eta,(\phi_i)) = 
\begin{pmatrix} 
\log\frac{\MA_{\theta_1+\eta_1}(\phi_1)}{\omega_\eta^n} + \sum \phi_i -\log|\alpha_1+[\eta_1]| \\
\vdots \\
\log \frac{\MA_{\theta_k+\eta_k}(\phi_k)}{\omega_\eta^n} + \sum \phi_i -\log|\alpha_k+[\eta_k]| \\
\int_X \phi_2 \omega_0^n \\
\vdots \\
\int_X \phi_k \omega_0^n
\end{pmatrix}.
$$
\end{definition}

Note that by Lemma~\ref{lemma:MAeq} the $k$-tuple $(\theta_i+\eta_i+dd^c\phi_i)$ representing the decomposition $(\alpha_i+[\eta_i])$ is cKE if and only if $F(\eta,(\phi_i+C_i))=0$ for some constants $C_1,\ldots,C_k$.

\begin{lem}
\label{lemma:invertibility}
Assume $X$ is Fano K\"ahler-Einstein, $\Aut(X)$ is discrete and $(\alpha_i)=(\lambda_i c_1(X))$ is a parallel decomposition of $c_1(X)$. Let $\theta_i=\lambda_i\omega_{KE}$ where $\omega_{KE}$ is the unique K\"ahler-Einstein metric on $X$. Then the linearization of $F$ at $(0,0)$ with respect to the second argument is invertible.
\end{lem}
\begin{proof}
The linearization of $F$ at $(0,0)$ with respect to the second argument is given by
$$
H(v_i) = \begin{pmatrix} -\frac{1}{\lambda_1}\Delta v_1 + \sum v_i \\
\vdots \\
-\frac{1}{\lambda_k}\Delta v_k+\sum v_i \\
\int_X v_2 \omega_0^n \\
\vdots \\
\int_X v_k \omega_0^n
\end{pmatrix}
$$
where $\Delta:C^{2,1/2}\rightarrow C^{0,1/2}$ is the Laplace-Beltrami operator associated to $\omega_{KE}$. This follows from the fact that
$$ \log\frac{\MA_{\theta_i}(\phi_i)}{\omega_\eta^n} = \log\frac{\MA_{\theta_i}(\phi_i)}{\theta_i^n} + \log \frac{\theta_i^n}{\omega_\eta^n} $$ 
is Fr\'echet differentiable in $\phi$ and its derivative is given by the negative of the Laplace-Beltrami operator $\Delta_{\theta_i}$ associated to the metric $\theta_i$. By assumption $\theta_i=\lambda_i\omega_{KE}$ and by $-1$-homogenity of $\Delta_{\theta_i}$ with respect to $\theta_i$ we have 
$$ \Delta_{\theta_i} = \Delta_{\lambda_i\omega_{KE}} = \frac{1}{\lambda_i}\Delta_{\omega_{KE}} = \frac{1}{\lambda_i}\Delta. $$

To see that $H$ is invertible we will perform a change of basis. The proposition will then follow from the well known fact that the Laplace-Beltrami operator of a K\"ahler-Einstein metric is invertible (modulo constants) and its smallest eigenvalue is, as long as $\Aut(X)$ is discrete, larger than 1.

Let $(w_i)=H(v_i)$. Put 
$$ w_1' = \sum_i \lambda_i w_i. $$
Moreover, put
$$ w_i'=w_i-w_{i-1} $$
for $1<i\leq k$ and 
$$ w_i' = w_i $$ 
for $k<i\leq 2k-1$.
%
%
Then, since $\sum\lambda_i=1$, 


$$ 
\begin{pmatrix} 
w_1'\vphantom{\left(\sum_i v_i\right)} \\
w_2'\vphantom{\left(\frac{v_2}{\lambda_2}\right)} \\
\vdots \\
w_k'\vphantom{\left(\frac{v_2}{\lambda_2}\right)} \\
w_{k+1}'\vphantom{\left(\frac{v_2}{\lambda_2}-\frac{v_1}{\lambda_1}\right)} \\
\vdots \\
w_{2k-1}'\vphantom{\left(\frac{v_k}{\lambda_k}-\frac{v_{k-1}}{\lambda_{k-1}}\right)} \\
\end{pmatrix}
=
\begin{pmatrix} 
-\Delta \left(\sum_i v_i\right) + \sum_i v_i \\
-\Delta \left(\frac{v_2}{\lambda_2}-\frac{v_1}{\lambda_1}\right) \\
\vdots \\
-\Delta \left(\frac{v_k}{\lambda_k}-\frac{v_{k-1}}{\lambda_{k-1}}\right) \\
\int_X v_2\vphantom{\left(\frac{v_2}{\lambda_2}-\frac{v_1}{\lambda_1}\right)} \omega_0^n \\
\vdots \\
\int_X v_k\vphantom{\left(\frac{v_k}{\lambda_k}-\frac{v_{k-1}}{\lambda_{k-1}}\right)} \omega_0^n \\
\end{pmatrix}
=
\begin{pmatrix} 
-\Delta v_1' + v_1' \vphantom{\left(\sum_i v_i\right)} \\
-\Delta v_2' \vphantom{\left(\frac{v_2}{\lambda_2}\right)} \\
\vdots \\
-\Delta v_k' \vphantom{\left(\frac{v_2}{\lambda_2}\right)} \\
\int_X v_2 \omega_0^n\vphantom{\left(\frac{v_2}{\lambda_2}-\frac{v_1}{\lambda_1}\right)} \\
\vdots \\
\int_X v_k \omega_0^n \vphantom{\left(\frac{v_k}{\lambda_k}-\frac{v_{k-1}}{\lambda_{k-1}}\right)}
\end{pmatrix}
$$
where 
$$ v_1' = \sum_i v_i $$
and 
$$ v_i' = \frac{v_i}{\lambda_i}-\frac{v_{i-1}}{\lambda_{i-1}} $$
for $1<i\leq k$.

It follows from the properties of $\Delta$ above that the map $(v_i')\mapsto (w_i')$, and hence the map $(v_i)\mapsto (w_i)$, is invertible. This proves the lemma. 
\end{proof}

\begin{proof}[Proof of Theorem~C]
Assume $(\alpha_i)=(\lambda_i c_1(X))$. Let $\theta_i=\lambda_i\omega_{KE}$ where $\omega_{KE}$ is the unique K\"ahler-Einstein metric on $X$. Then $(\theta_i)$ is cKE and hence $F(0,0)=0$. We want to apply the implicit function theorem to show that for any $\eta\in H_k$ close to 0, there is $(\phi_i)\in \Pi \PSH(\theta_i+\eta_i)\cap C^{2,1/2}$ such that $F(\eta,(\phi_i))=0$. This would imply that there are $(\theta_i+\eta_i+dd^c\phi_i)$, representing the decomposition 
\begin{equation} 
(\alpha_i+[\eta_i]), 
\label{eq:pertdec}
\end{equation}
is cKE. Since any decomposition close to $(\alpha_i)$ may be written as \eqref{eq:pertdec} for some small $\eta\in H_k$, this proves the theorem. 

Now, by Lemma~\ref{lemma:invertibility} the only thing we need to verify to apply the implicit function theorem is that $F$ is continuously Fr\'echet differentiable with respect to the first argument. To do this we need to show that the map
$$ \eta\rightarrow \omega_\eta $$
is continuously Fr\'echet differentiable. But this follows from the standard application of the inverse function theorem used in the proof of the Calabi-Yau Theorem. See for example Section 5 in \cite{Tian}.
\end{proof}

\end{document}